\documentclass[11pt,a4paper,reqno]{amsart}
\usepackage{graphicx}
\usepackage{amssymb}
\usepackage{cite}
\usepackage{amsmath}
\usepackage{latexsym}
\usepackage{amscd}
\usepackage{amsthm}
\usepackage{mathrsfs}
\usepackage{url}
\usepackage[linktocpage=true,colorlinks,citecolor=blue,linkcolor=blue,urlcolor=blue]{hyperref}

\usepackage[centertags]{amsmath}
\usepackage{amsfonts}
\usepackage{dsfont}
\usepackage{wasysym}
\usepackage[draft]{optional}
\usepackage[usenames]{color}
\newtheorem{thm}{Theorem}[section]
\newtheorem{corr}[thm]{Corollary}
\newtheorem{lem}[thm]{Lemma}
\newtheorem{prop}[thm]{Proposition}

\theoremstyle{definition}

\theoremstyle{remark}

\newtheorem{rem}{Remark}[section]
\numberwithin{equation}{section}
\def\R{\mathbb R}

\def\SS{\mathbb S}

\def\f{\frac}

\def\ra{\rightarrow}
\def\pt{\partial}

\newcommand{\twopartdef}[4]
{
	\left\{
		\begin{array}{ll}
			#1 & \mbox{if  } #2 \bigskip \\
			#3 & \mbox{if  } #4
		\end{array}
	\right.
}

\begin{document}
\title[Smooth compactness of $f$-minimal hypersurfaces]{Smooth compactness of $f$-minimal hypersurfaces  with bounded $f$-index}
\author{Ezequiel Barbosa}
\address{Universidade Federal de Minas Gerais (UFMG), Departamento de Matem\'{a}tica, Caixa Postal 702, 30123-970, Belo Horizonte, MG, Brazil}
\email{ezequiel@mat.ufmg.br}
\author{Ben Sharp}
\address{Department of Mathematics, South Kensington Campus, Imperial College London, LONDON,
SW7 2AZ, United Kingdom}
\email{ben.g.sharp@gmail.com}
\author{Yong Wei}
\address{Department of Mathematics, University College London, Gower Street, London,WC1E 6BT, United Kingdom}
\email{yong.wei@ucl.ac.uk}
\thanks{The first author was supported by FAPEMIG and CNPq grants; The second author was supported by Andr\'e Neves' ERC and Leverhulme trust grants; The third author was supported by Jason D Lotay's EPSRC grant}
\date{\today}
\maketitle

\begin{abstract}
Let $(M^{n+1},g,e^{-f}d\mu)$ be a complete smooth metric measure space with $2\leq n\leq 6$ and Bakry-\'{E}mery Ricci curvature bounded below by a positive constant. We prove a smooth compactness theorem for the space of complete embedded $f$-minimal hypersurfaces in $M$ with uniform upper bounds on $f$-index and weighted volume. As a corollary, we obtain a smooth compactness theorem for the space of embedded self-shrinkers in $\R^{n+1}$ with $2\leq n\leq 6$. We also prove some estimates on the $f$-index of $f$-minimal hypersurfaces, and give a conformal structure of $f$-minimal surface with finite $f$-index in three-dimensional smooth metric measure space.
\end{abstract}

\section{Introduction}\label{sec:intro}
Let $(M^{n+1},g)$ be a complete smooth Riemannian manifold and $f$ be a smooth function on $M$. We denote by $\overline{\nabla}$ the Levi-Civita connection on $M$ with respect to $g$. The Bakry-\'{E}mery Ricci curvature of $M$ is defined by
\begin{equation}
Ric_f=Ric+\overline{\nabla}^2f,
\end{equation}
which is an important generalization of the Ricci curvature. The equation $Ric_f=\kappa g$ for some constant $\kappa$ is just the gradient Ricci soliton equation, which palys an important role in the singularity analysis of Ricci flow (see \cite{Cao}). Denote by $d\mu$ the volume form on $M$ with respect to $g$, then $(M^{n+1},g,e^{-f}d\mu)$ is often called a smooth metric measure space. We refer the interested readers to \cite{Wei-Wylie} for further motivation and examples of smooth metric measure spaces.

Suppose that $\Sigma$ is a hypersurface in $M$. Let $x\in \Sigma$ and $\nu$ be the unit normal vector at $x$. The mean curvature $H$ of $\Sigma$ is defined as $H=\sum_{i=1}^n\langle\overline{\nabla}_{e_i}\nu,e_i\rangle$ for an orthonormal basis $\{e_i\}_{i=1}^n\subset T\Sigma$. Define the $f$-mean curvature at $x\in \Sigma$ by
 \begin{align}
    H_f=H-\langle\overline{\nabla}f,\nu\rangle.
 \end{align}
Then $\Sigma$ is called an $f$-minimal hypersurface in $M$ if its $f$-mean curvature $H_f$ vanishes everywhere. The most well known example of smooth metric measure space is the Gaussian soliton: $(\mathbb{R}^{n+1},g_0,e^{-\f 14|x|^2}d\mu)$ whose Bakry-\'{E}mery Ricci curvature satisfies $Ric_f=\f 12g_0$, where $g_0$ is the standard Euclidean metric on $\R^{n+1}$. The $f$-minimal hypersurface in the Gaussian soliton is the self-shrinker $\Sigma^{n}\subset\R^{n+1}$ which satisfies
\begin{equation}\label{self-shr}
  H=\f 12\langle x,\nu\rangle.
\end{equation}
Self-shrinkers play an important role in the mean curvature flow, as they correspond to the self-similar solutions to mean curvature flow, and also describe all possible blow ups at a given singularity. The simplest examples of self-shrinkers in $\R^{n+1}$ are the flat hyperplanes $\R^n$ through the origin, the spheres centred at the origin $\SS^n(\sqrt{2n})$ and the cylinders $\SS^k(\sqrt{2k})\times\R^{n-k}$, $1\leq k\leq n-1$. Angenent \cite{Ang} has constructed a compact embedded self-shrinker in $\R^3$ of genus one, which is often called the Angenent torus. Recently, Brendle \cite{brendle} proved that the only compact embedded self-shrinker in $\R^3$ of genus zero is the round sphere $\SS^2(2)$. See \cite{Drug,Drug-K,Kap,Mol,Ngu} for further examples of self-shrinkers.

In 2012, Colding-Minicozzi \cite{CM09} proved the following smooth compactness theorem for the space of embedded self-shrinkers.
\begin{thm}[\cite{CM09}]\label{thm-CM09}
Given an integer $\gamma\geq 0$ and a constant $0<\Lambda<\infty$, the space of smooth complete embedded self-shrinkers $\Sigma\subset \R^{3}$ with
\begin{itemize}
  \item genus at most $\gamma$,
  \item $\pt\Sigma=\emptyset$,
  \item $Area(B_R(x_0)\cap\Sigma)\leq \Lambda R^2$ for all $x_0\in\R^{3}$ and all $R>0$
\end{itemize}
is compact in the sense that any sequence in this class contains a subsequence that converges in the $C^m$ topology on compact subsets for any $m\geq 2$.
\end{thm}

The surfaces $\Sigma$ in the above theorem are each assumed to be homeomorphic to some  closed surface $S_\Gamma$ of genus $\Gamma\leq \gamma$ with a finite number of points removed $\Sigma \cong S_\Gamma \setminus \{p_1,\dots p_k\}$ - thus each surface has a finite number of ends. There is no uniform control on $k$, or equivalently the number of ends so the global topology of the limiting surface is \emph{a-priori} not known.
 
The third condition in Theorem \ref{thm-CM09} means that the self-shrinkers have at most Euclidean volume growth. Generally, we say that a self-shrinker $x:\Sigma^{n}\ra\R^{n+1}$ has at most Euclidean volume growth if $Area(B_R(x_0)\cap\Sigma)\leq \Lambda R^n$ where $B_R(x_0)$ is the Euclidean ball of radius $R$ centred at $x_0\in\R^{n+1}$. If the self-shrinker arises as the singularity of mean curvature flow,  Huisken's monotonicity formula implies that it has at most Euclidean volume growth, see \cite[Corollary 2.13]{CM2012}; furthermore Lu Wang \cite{Lu} proved that the entire graphical self-shrinker has at most Euclidean volume growth; and finally Ding-Xin \cite{Ding-Xin} then proved that any properly-immersed self-shrinker has at most Euclidean volume growth. Since self-shrinkers model singularities of mean curvature flow, Theorem \ref{thm-CM09} can be viewed as a compactness result for the space of all singularities. In particular, Theorem \ref{thm-CM09} plays an important role in understanding generic mean curvature flow in \cite{CM2012}.

Later, Ding-Xin \cite{Ding-Xin} proved a related compactness theorem for the space of self-shrinkers in $\R^3$. After that, based on the observation that a self-shrinker is a special example of an $f$-minimal surface, the third author and H. Li \cite{Li-Wei} proved a compactness theorem for the space of closed $f$-minimal surfaces in a closed three dimensional smooth metric measure space with positive Ricci curvature, generalizing the classical compactness theorem of closed minimal surfaces in a closed three manifold with positive Ricci curvature by Choi and Schoen \cite{Choi-1}. The results in \cite{Li-Wei,CM09,Ding-Xin} were later generalised by Cheng-Mejia-Zhou \cite{CMZ,CMZ2} to complete $f$-minimal surfaces in complete 3-dimensional smooth metric measure spaces. On the other hand the second author \cite{b-sharp} proved a smooth compactness result for closed embedded minimal hypersurfaces with bounded index and volume in a closed Riemannian manifold $M^{n+1}$ with positive Ricci curvature and $2\leq n\leq 6$, generalising Choi-Schoen's \cite{Choi-1} compactness theorem to higher dimensions.
\begin{thm}[\cite{b-sharp}]\label{thm-ben}
Let $(M^{n+1},g)$ be a closed Riemannian manifold with $Ric>0$ and $2\leq n\leq 6$. Given an integer $I$ and a constant $0<\Lambda<\infty$, the space of smooth closed embedded minimal hypersurfaces $\Sigma\subset M^{n+1}$ satisfying $Vol(\Sigma)\leq \Lambda$ and $\textrm{index} (\Sigma)\leq I$ is compact in the smooth topology.
\end{thm}
Note that Ejiri-Micallef \cite{Eiji-M} proved that for closed minimal surface $\Sigma$ of genus $\gamma$ in a Riemannian manifold $M^m$, the area index satisfies
\begin{equation}\label{indx-Ejiri-Mic}
  ind(\Sigma)\leq C(m,M)(Area(\Sigma)+\gamma-1),
\end{equation}
where $C(m,M)$ is a constant depending  on $m,M$.  The main interest in the above result is that no assumption is made on the second fundamental form of $\Sigma\subset M^m$. Moreover, when $n=2$ the area is uniformly bounded form above by the genus and the lower Ricci-bound by an estimate of Choi-Wang \cite{CW83}. Thus Theorem \ref{thm-ben} is a natural generalisation of Choi-Schoen's compactness theorem. In this paper, we borrow the idea in \cite{b-sharp} to prove the following smooth compactness theorem for the space of $f$-minimal hypersurfaces.
\begin{thm}\label{main-thm}
Let $(M^{n+1},g,e^{-f}d\mu)$ be a complete smooth metric measure space with $2\leq n\leq 6$ and $Ric_f\geq \kappa>0$ for some positive constant $\kappa$. Given an integer $I$ and a constant $0<\Lambda<\infty$, the space $S_{\Lambda,I}$ of smooth complete embedded $f$-minimal hypersurfaces $\Sigma\subset M^{n+1}$ with
\begin{itemize}
  \item $f$-index at most $I$,
  \item $\pt\Sigma=\emptyset$,
  \item $\int_{\Sigma}e^{-f}dv\leq \Lambda$
\end{itemize}
is compact in the smooth topology. Namely, any sequence of $S_{\Lambda,I}$ has a subsequence that converges in $C^k$ topology on compact subsets to a hypersurface in $S_{\Lambda,I}$ for any $k\geq 2$.
\end{thm}
Our Theorem \ref{main-thm} generalises  Theorem \ref{thm-ben}, as when $f=const$, an $f$-minimal hypersurface is just a minimal hypersurface with $Ric_f=Ric$. Note that an $f$-minimal hypersurface $\Sigma$ in $(M^{n+1},g,e^{-f}d\mu)$ is also minimal in $(M^{n+1}, \tilde{g})$, where $\tilde{g}=e^{-\frac 2nf}g$ is a conformal metric of $g$, see \S \ref{conf-chang}. However, Theorem \ref{main-thm} cannot follow from Theorem \ref{thm-ben} directly by a conformal transformation, because the condition $Ric_f\geq \kappa>0$ cannot guarantee the Ricci tensor of the conformal metric $\tilde{g}$ has a positive lower bound $\widetilde{Ric}\geq c>0$.  Moreover, the Riemannian manifold $(M,g)$ with $Ric\geq \kappa>0$ must be compact by Myers' theorem, while this is not true for a smooth metric measure space $(M^{n+1},g,e^{-f}d\mu)$ with $Ric_f\geq \kappa>0$. In Theorem \ref{main-thm}, the $f$-index of an $f$-minimal hypersurface is defined as the dimension of the subspace such that the second variation of the weighted volume is negative, see \S \ref{sec:vari} for details. The proof of Theorem \ref{main-thm} will be given in \S \ref{sec:proof}.
Since a self-shrinker is an $f$-minimal hypersurface in the Gaussian soliton which has constant $Ric_f=\frac 12$, Theorem \ref{main-thm} has the following immediate corollary, which can be viewed as a generalisation of Theorem \ref{thm-CM09} to higher dimensions.
\begin{corr}\label{main-corr}
For $2\leq n\leq 6$, given an integer $I$ and a constant $0<\Lambda<\infty$, the space of smooth complete embedded self-shrinkers $\Sigma\subset \R^{n+1}$ with
\begin{itemize}
  \item $f$-index at most $I$,
  \item $\pt\Sigma=\emptyset$,
  \item $Vol(B_R(x_0)\cap\Sigma)\leq \Lambda R^n$ for all $x_0\in\R^{n+1}$ and all $R>0$
\end{itemize}
is compact in smooth topology.
\end{corr}

\begin{rem}
The relationship between Corollary \ref{main-corr} and Theorem \ref{thm-CM09} is not well understood presently, in the sense that it is not clear which is stronger than the other. However, we expect there to be a more explicit relationship between the index, genus and number of ends for self-shrinkers (or more generally $f$-minimal surfaces). We mention some partial results below for closed surfaces which follow easily from the estimates available for minimal surfaces in Riemannian manifolds - the relationship in this case is comparatively well understood. 
\end{rem}

Note that the third condition in Corollary \ref{main-corr} implies that the weighted volume $\int_{\Sigma}e^{-\frac{|x|^2}4}dv\leq \tilde{\Lambda}$ for some positive constant $\tilde{\Lambda}$. Then the conclusion of Corollary \ref{main-corr} follows directly from Theorem \ref{main-thm}.

In particular for closed self shrinkers we have the following which follows easily from the estimates for minimal surfaces: 
\begin{prop}\label{mainprop-indexup-ss1}
Let $\Sigma^n \subset\R^{n+1}$ be a closed smooth embedded self-shrinker satisfying $|x|\leq c_1$, i.e., $\Sigma$ is contained in an Euclidean ball of radius $c_1$. Then the $f$-index of $\Sigma$ satisfies
\begin{equation}\label{indx-upperbd-selfs1}
  ind_f(\Sigma)\leq \twopartdef{C(n,c_1)\int_{\Sigma}(|A|^2+\frac 12)^{\frac n2}e^{-\frac{|x|^2}4}dv}{n\geq 3}{C(c_1)(\int_{\Sigma}e^{-\f{|x|^2}{4}}dv+\gamma-1)}{n=2,}
\end{equation}
where $C(n,c_1)$ is a uniform constant depending on $n,c_1$.
\end{prop}
\begin{rem}
Cao-Li \cite{Cao-Li} proved that any compact self-shrinker $\Sigma^n \subset\R^{n+1}$ satisfying $H^2=\frac 14|x^{\bot}|^2\leq n/2$ is the sphere $\SS^n(\sqrt{2n})$, see also \cite{C-Espinar,Li-Wei-self-shr,Pigola-rimo} for related results.
\end{rem}

We moreover have the analogous result for general closed $f$-minimal surfaces in smooth metric measure spaces, from which Proposition \ref{indx-upperbd-selfs1} follows; 
\begin{prop}\label{1.8}
Let $\Sigma^n$ be a smooth closed embedded $f$-minimal hypersurface in $(M^{n+1}, g, e^{-f}d\mu)$. Suppose that $\Sigma\subset\subset K$ where $K$ is compact with open interior, and $\partial K$ is smooth. Then the $f$-index of $\Sigma$ satisfies 
\begin{equation}\label{indx-upperbd-selfs2}
  ind_f(\Sigma)\leq \twopartdef{C(n,K,M,f)\int_{\Sigma}\max\{1,|A|^2+Ric_f(\nu,\nu)\}^{\frac n2}e^{-f}dv}{n\geq 3}{C(K,M,f)(\int_{\Sigma}e^{-f}dv+\gamma-1)}{n=2,}
\end{equation}
where $C(n,K,M,f)$ is a uniform constant depending on $n,K,f$ and $M$. 
\end{prop}

\begin{rem}
We note that since we are working in a compact region of $M$, the quantities on the right hand side of \eqref{indx-upperbd-selfs1}, \eqref{indx-upperbd-selfs2} and   are equivalent (up to a constant depending only on $K,c_1$ and $f$) to the $(M,g)$-area $\int_\Sigma 1 dv$, plus the usual (un-weighted) total curvature $\int_\Sigma |A|^n d v$.  
\end{rem}

From Corollary 8 in \cite{Li-Wei}, we know that for a closed embedded $f$-minimal surface $\Sigma$ in a closed smooth metric measure space $(M^{3},g,e^{-f}d\mu)$ with $Ric_f\geq \kappa>0$,  the weighted area of $\Sigma$ satisfies
\begin{equation}\label{weighted-areabound}
\int_{\Sigma}e^{-f}dv\leq \frac {16\pi}{\kappa}(\gamma+1)e^{-\min_\Sigma f}.
\end{equation}
Combining \eqref{indx-upperbd-selfs2}, \eqref{weighted-areabound} and Theorem \ref{main-thm}, we recover Li-Wei's \cite{Li-Wei} theorem.
\begin{thm}[\cite{Li-Wei}]\label{li-wei}
Let $(M^{3},g,e^{-f}d\mu)$ be a closed (i.e. compact with empty boundary) smooth metric measure space with $Ric_f\geq \kappa>0$. Then the space of closed embedded $f$-minimal surfaces of fixed topological type in $(M^{3},g,e^{-f}d\mu)$ is compact in the smooth topology.
\end{thm}
Further to the above we know that if $\Sigma^2$ is a closed embedded $f$-minimal surface in a complete non-compact smooth metric
measure space $(M^3,g,e^{−f}d\mu)$ with $Ric_f \geq k>0$, and $\Sigma\subset D$ where $D$ is some compact domain with $g$-convex boundary, then by \cite[Proposition 8]{CMZ2} we have that \eqref{weighted-areabound} is satisfied. Combining this with \eqref{indx-upperbd-selfs2}, we consider the compactness for $f$-minimal surfaces in a complete non-compact manifold and recover Cheng-Meija-Zhou's \cite{CMZ2} theorem.
\begin{thm}[\cite{CMZ2}]\label{cmz}
Let $(M^3,g,e^{−f}d\mu)$ be a complete non-compact smooth metric
measure space with $Ric_f \geq k$, where $k$ is a positive constant. Assume that
$M^3$ admits an exhaustion by bounded domains with convex boundary. Then
the space, denoted by $S_{D,g}$, of closed embedded $f$-minimal surfaces in $M^3$ with
genus at most $g$ and diameter at most $D$ is compact in the $C^m$ topology, for
any $m\geq2$.
\end{thm}

\begin{rem}
We would like to point out that Theorem \ref{main-thm} can now be thought of as an extension of Theorems \ref{li-wei} and \ref{cmz} to higher dimensions. 

\end{rem}

In the last section, we will provide information on the conformal structure of $f$-minimal surfaces with finite $f$-index in $(M^{3},g,e^{-f}d\mu)$ with nonnegative $V$, where $V$ is a quantity defined by (see \cite{Espinar})
\begin{equation}\label{V-def-1}
  V=\frac 13\left(\frac 12R_f^P+\frac 12|A|^2+\frac 18|\nabla f|^2\right),
\end{equation}
and here $R_f^P=R+2\overline{\Delta}f-|\overline{\nabla}f|^2$ is the Perelman's scalar curvature (see \cite{perel}) of $(M^{3},g,e^{-f}d\mu)$, and $\nabla$ is the Levi-Civita connection on $\Sigma$ with respect to the induced metric from $g$.
\begin{thm}\label{mainthm-fisch}
Let $\Sigma$ be a complete orientable $f$-minimal surface of finite $f$-index in $(M^{3},g,e^{-f}d\mu)$ with $V\geq 0$, then $\Sigma$ is conformally diffeomorphic to a Riemannian surface with a finite number of points removed.
\end{thm}
Theorem \ref{mainthm-fisch} is a generalisation of Fischer-Colbrie's \cite{fisch} result, where she proved that a complete orientable minimal surface $\Sigma$ of finite index in a three manifold $(M^3,g)$ with nonnegative scalar curvature is conformally diffeomorphic to a Riemannian surface with a finite number of points removed. We further remark that for complete self shrinkers $\Sigma^2$ in $\R^3$ we will not have $V\geq 0$ in general.

\section{$f$-minimal hypersurfaces}\label{sec:prelim}

In this section, we collect and prove some properties about $f$-minimal hypersurfaces. The readers can refer to \cite{CMZ,CMZ2,Cheng-Zhou,Espinar,Impera-Ri,Li-Wei,Liu} for more information about $f$-minimal hypersurfaces.

\subsection{Variational properties}\label{sec:vari}
Let $(M^{n+1},g,e^{-f}d\mu)$ be a complete smooth metric measure space. Suppose that $\Sigma$ is a smooth embedded hypersurface in $M^{n+1}$. For any normal variation $\Sigma_s$ of $\Sigma$ with compactly supported variation vector field $X\in T^{\bot}\Sigma$, the first variation formula for the weighted volume of $\Sigma$
\begin{equation*}
  Vol_f(\Sigma):=\int_{\Sigma}e^{-f}dv
\end{equation*}
is given by
\begin{equation}\label{1st-vari}
 \frac d{ds}\bigg|_{s=0}Vol_f(\Sigma_s)=-\int_{\Sigma}g(X,\textbf{H}_f)e^{-f}dv
\end{equation}
where $dv$ is the volume form of $\Sigma$ with respect to the metric induced from $(M,g)$. We say that $\Sigma$ is $f$-minimal if $\frac d{ds}\big|_{s=0}Vol_f(\Sigma_s)=0$ for any compactly supported normal variation $\Sigma_s$. Then from \eqref{1st-vari}, $\Sigma$ is $f$-minimal if and only if $\textbf{H}_f=0$ on $\Sigma$.

We also have the second variation formula for $Vol_f(\Sigma)$ (see e.g.,\cite{CMZ,Liu}):
\begin{align}\label{2nd-vari}
 \frac {d^2}{ds^2}\bigg|_{s=0}Vol_f(\Sigma_s)=&\int_{\Sigma}\left(|\nabla^{\bot}X|^2-(|A|^2|X|^2+Ric_f(X,X))\right)e^{-f}dv
\end{align}
where $\nabla^{\bot}$ is the connection on the normal bundle $T^{\bot}\Sigma$, $|A|^2$ is the squared norm of the second fundamental form $A$ of $\Sigma$ in $(M,g)$ and $Ric_f$ is the Bakry-\'{E}mery Ricci tensor of $(M,g,e^{-f}d\mu)$. We associate  the right hand side of \eqref{2nd-vari} to a bilinear form on the space of compactly supported normal vector field on $\Sigma$
\begin{align}\label{Qf-X}
Q_f^{\Sigma}(X,X)=&\int_{\Sigma}\left(|\nabla^{\bot}X|^2-(|A|^2|X|^2+Ric_f(X,X))\right)e^{-f}dv.
\end{align}
An $f$-minimal hypersurface $\Sigma$ is called $f$-stable if $Q_f^{\Sigma}(X,X)\geq 0$ for any compactly supported $X\in T^{\bot}\Sigma$. The $f$-index of $\Sigma$, denoted by $ind_f(\Sigma)$, is defined to be the dimension of the subspace such that $Q_f^{\Sigma}$ is strictly negative definite.

If $\Sigma$ is two-sided, then there exists a globally defined unit normal $\nu$ on $\Sigma$.  Then $f$-mean curvature vector field of $\Sigma$ in $(M,g)$ satisfies $\textbf{H}_f=-H_f\nu=0$. Set $X=\varphi\nu$ for some smooth compactly supported function $\varphi$ on $\Sigma$. Then the quadratic form \eqref{Qf-X} is equivalent to
\begin{align}\label{Qf}
Q_f^{\Sigma}(\varphi,\varphi)=&\int_{\Sigma}\left(|\nabla \varphi|^2-(|A|^2+Ric_f(\nu,\nu))\varphi^2\right)e^{-f}dv\nonumber\\
 =&-\int_{\Sigma}\varphi L_f\varphi e^{-f}dv,
\end{align}
where $\nabla$ is the Levi-Civita connection on $\Sigma$ w.r.t the metric induced from $(M,g)$ - thus in particular there are no two-sided $f$-minimal surfaces which have zero index (or are stable). The operator $L_f$ is defined by
\begin{equation}\label{def-Lf}
  L_f=\Delta_f+|A|^2+Ric_f(\nu,\nu),
\end{equation}
where $\Delta_f=\Delta-\nabla f\cdot\nabla$ is the $f$-Laplacian on $\Sigma$, which is self-adjoint in the weighted $L^2$ sense, i.e., for any two $u,v\in C_c^{\infty}(\Sigma)$,
\begin{equation}\label{self-aj}
  \int_{\Sigma}u\Delta_fve^{-f}dv=-\int_{\Sigma}\langle\nabla u,\nabla v\rangle e^{-f}dv.
\end{equation}
Then $L_f$ is also self-adjoint in the weighted $L^2$ sense. For any compact domain $\Omega\subset\Sigma$, consider the following Dirichlet eigenvalue problem
\begin{equation}\label{dirichlet-pb}
-L_fu=\lambda u,\quad u|_{\pt\Omega}=0,\quad u\in C^{\infty}(\Omega).
\end{equation}
The standard spectrum theory implies that the set of eigenvalues of the problem \eqref{dirichlet-pb} is a discrete increasing sequence $\lambda_1(\Omega)<\lambda_2(\Omega)\leq\cdots\ra\infty$. There is some number $k$, counted with multiplicity, such that $\lambda_1,\cdots,\lambda_k<0\leq \lambda_{k+1}$. The Morse index of $L_f$ on $\Omega$ is defined to be the number $k$ so that $\lambda_{k+1}$ is the first nonnegative eigenvalue of $L_f$ on $\Omega$, which we denote by $ind_f(\Omega)$. By the minimax characterisation of eigenvalues we have 
\begin{equation*}
  ind_f(\Omega)\leq ind_f(\Omega'),\textrm{ if }\Omega\subset \Omega'\subset \Sigma.
\end{equation*}
As in the complete two-sided minimal hypersurface case (see \cite{fisch}), the $f$-index of complete two-sided $f$-minimal hypersurface $\Sigma$ is equal to the supremum over compact domains of $\Sigma$ of the $f$-index of compact domain $\Omega\subset\Sigma$, i.e., $ind_f(\Sigma)=\sup\limits_{\Omega\subset\Sigma}ind_f(\Omega)$.

\subsection{An $f$-minimal surface as a minimal surface in a conformally changed metric}\label{conf-chang}
The $f$-minimal hypersurface can be viewed as a minimal hypersurface under a conformal metric. Define a conformal metric $\tilde{g}=e^{-\frac 2{n}f}g$ on $M$. Then the volume of $\Sigma$ with respect to the induced metric from $\tilde{g}$ is
\begin{equation}\label{f-area}
  \widetilde{Vol}(\Sigma)=\int_{\Sigma}d\tilde{v}=\int_{\Sigma}e^{-f}dv=Vol_f(\Sigma).
\end{equation}
The first variation formula for $\widetilde{Vol}$ is given by
\begin{align}\label{1st-vari-2}
  \frac d{ds}\bigg|_{s=0}\widetilde{Vol}(\Sigma_s)=&-\int_{\Sigma}\tilde{g}(X,\tilde{\textbf{H}})d\tilde{v}
  =-\int_{\Sigma}e^{-\frac 2{n}f}g(X,\tilde{\textbf{H}})e^{-f}dv,
\end{align}
where $\tilde{\textbf{H}}$ is the mean curvature vector field of $\Sigma$ in $(M,\tilde{g})$. Comparing \eqref{1st-vari} and \eqref{1st-vari-2}, and noting \eqref{f-area}, we conclude that $\tilde{\textbf{H}}=e^{\frac 2{n}f}\textbf{H}_f$ and $\Sigma$ is $f$-minimal in $(M,g)$ if and only if $\Sigma$ is minimal in $(M,\tilde{g})$.

When $\Sigma$ is two-sided, there exists a globally defined unit normal $\tilde{\nu}$ on $\Sigma$. Then $\tilde{\nu}=e^{\frac f{n}}\nu$, where $\nu$ is the unit normal of $\Sigma$ in $(M,{g})$. The mean curvature vector field of $\Sigma$ in $(M,\tilde{g})$ satisfies $\tilde{\textbf{H}}=-\tilde{H}\tilde{\nu}=-\tilde{H}e^{\frac f{n}}\nu$. Thus, we have
\begin{equation}\label{mean-cur}
  \tilde{H}=e^{\frac f{n}}H_f.
\end{equation}
The variation vector field $X$ can be rewritten as $X=\varphi\nu=e^{-\frac fn}\varphi\tilde{\nu}$. The second variation formula for $\widetilde{Vol}(\Sigma)$ is given by
\begin{align}\label{2st-vari-2}
  \frac {d^2}{ds^2}\bigg|_{s=0}\widetilde{Vol}(\Sigma_s)=&-\int_{\Sigma}\tilde{g}(X,\tilde{L}(X))d\tilde{v}
  =-\int_{\Sigma}e^{-\frac f{n}}\varphi \tilde{L}(e^{-\frac f{n}}\varphi)e^{-f}dv,
\end{align}
where $\tilde{L}$ is the Jacobi operator of $\Sigma\subset (M,\tilde{g})$:
\begin{equation}\label{td-L}
  \tilde{L}=\Delta_{\tilde{g}}+|\tilde{A}|_{\tilde{g}}^2+\widetilde{Ric}(\tilde{\nu},\tilde{\nu}).
\end{equation}

Since $\widetilde{Vol}(\Sigma)=Vol_f(\Sigma)$ and the normal direction of $\Sigma$ in $M$ is unchanged under the conformal change of the ambient metric,  we know that the second variation
\begin{equation*}
  \frac{d^2}{ds^2}\bigg|_{s=0}Vol_f(\Sigma_s)=\frac{d^2}{ds^2}\bigg|_{s=0}\widetilde{Vol}(\Sigma_s)
\end{equation*}
for any normal variation $\Sigma_s$ of $\Sigma$. Thus $\Sigma$ is $f$-stable in $(M,g)$ if and only if $\Sigma$ is stable in $(M,\tilde{g})$ in the usual sense, and the $f$-index $ind_f$ of $\Sigma\subset (M,g)$ is equal to the index of $\Sigma\subset (M,\tilde{g})$ denoted by $\widetilde{ind}$.

\subsection{The relationship of the geometric quantities from the different viewpoints}\label{sec:2-3}
We can relate the second fundamental form $\tilde{A}$ of $\Sigma$ in $(M,\tilde{g})$ to the second fundamental form $A$ of $\Sigma$ in $(M,g)$ as follows: Choose an orthonormal basis $e_1,\cdots,e_{n}$ for $T\Sigma$ with respect to the metric induced from $(M,g)$. Then $A(e_i,e_j)=-g(\overline{\nabla}_{e_i}e_j,\nu)$. Under the conformal metric $\tilde{g}$, $\{\tilde{e}_1=e^{\frac f{n}}e_1,\cdots,\tilde{e}_{n}=e^{\frac f{n}}e_{n}\}$ and $\tilde{\nu}=e^{\frac f{n}}\nu$ are the orthonormal basis for $T\Sigma$ and unit normal vector field of $\Sigma$ in $(M,\tilde{g})$.  Then a direct calculation gives that
\begin{equation}\label{connection}
  \tilde{\nabla}_XY=\overline{\nabla}_XY-\frac 1{n}df(X)Y-\frac 1{n}df(Y)X+\frac 1{n}g(X,Y)\overline{\nabla}f
\end{equation}
for any tangent vector fields $X,Y$ on $M$, where $\tilde{\nabla}$ is the Levi-Civita connection on $(M,\tilde{g})$. Thus the second fundamental form $\tilde{A}$ satisfies
\begin{align}\label{sff}
  \tilde{A}(\tilde{e}_i,\tilde{e}_j) =& -\tilde{g}(\tilde{\nabla}_{\tilde{e}_i}\tilde{e}_j,\tilde{\nu}) \nonumber\\
  = &-\tilde{g}(\overline{\nabla}_{\tilde{e}_i}\tilde{e}_j -\frac 1{n}df(\tilde{e}_i)\tilde{e}_j-\frac 1{n}df(\tilde{e}_j)\tilde{e}_i+\frac 1{n}g(\tilde{e}_i,\tilde{e}_j)\overline{\nabla}f,\tilde{\nu})\nonumber\\
  =&-\tilde{g}(\overline{\nabla}_{\tilde{e}_i}\tilde{e}_j +\frac 1{n}g(\tilde{e}_i,\tilde{e}_j)\overline{\nabla}f,\tilde{\nu})\nonumber\\
  =&-e^{-\frac 2{n}f}g(e^{\frac 2{n}f}\overline{\nabla}_{e_i}e_j+\frac 1{n}e^{\frac 2{n}f}g(e_i,e_j)\overline{\nabla}f,e^{\frac f{n}}\nu)\nonumber\\
  =&e^{\frac f{n}}\left(A(e_i,e_j)-\frac 1{n}g(e_i,e_j)g(\overline{\nabla}f,\nu)\right).
\end{align}
Letting $i=j$ and summing from $1$ to $n$ in \eqref{sff} can also give the relation \eqref{mean-cur}. From \eqref{sff}, we can also obtain
\begin{align}\label{A^2}
  |\tilde{A}|^2_{\tilde{g}}=&\sum_{i,j=1}^n(\tilde{A}(\tilde{e}_i,\tilde{e}_j))^2\nonumber\\
  =&e^{\frac{2f}n}\left(\sum_{i,j=1}^n(A(e_i,e_j))^2+\frac 1ng(\overline{\nabla}f,\nu)^2-\frac 2nHg(\overline{\nabla}f,\nu)\right)\nonumber\\
  =&e^{\frac{2f}n}\left(|A|_g^2-\frac 1nH^2\right),
\end{align}
where in the last equality we used the fact $H=g(\overline{\nabla}f,\nu)$, since $\Sigma$ is $f$-minimal in $(M,g)$.

We next relate the $\widetilde{Ric}(\tilde{\nu},\tilde{\nu})$ and $Ric_f(\nu,\nu)$. Under the conformal change of the metric $\tilde{g}=e^{-\frac 2{n}f}g$ on $M$, we know that the Ricci curvature changes by the following formula
\begin{align}\label{ric-conformal}
  \widetilde{Ric} =& (n-1)\left(\frac{\overline{\nabla}^2f}n+\frac 1{n^2}df\otimes df-\frac 1{n^2}|\overline{\nabla} f|^2g\right) +\frac 1{n}\overline{\Delta} fg+Ric,
\end{align}
where $\overline{\nabla}, \overline{\Delta}$ are Levi-Civita connection and Laplacian operator on $(M,g)$. Then
\begin{align*}
  \widetilde{Ric}(\tilde{\nu},\tilde{\nu}) =&e^{\frac{2f}n}\biggl(\frac{n-1}{n^2}(n\overline{\nabla}^2f(\nu,\nu) +g(\overline{\nabla}f,\nu)^2-|\overline{\nabla} f|^2)+\frac 1{n}\overline{\Delta} f+Ric(\nu,\nu)\biggr)\\
  =&e^{\frac{2f}n}\biggl(Ric_f(\nu,\nu)-\frac{1}n\overline{\nabla}^2f(\nu,\nu) +\frac{n-1}{n^2}(H^2-|\overline{\nabla} f|^2) +\frac 1{n}\overline{\Delta} f\biggr)\\
  =&e^{\frac{2f}n}\biggl(Ric_f(\nu,\nu)+\frac{1}n(\Delta f+H^2) -\frac{n-1}{n^2}|\nabla f|^2\biggr),
\end{align*}
where in the last equality we used the elementary fact
\begin{equation*}
  \Delta f=\overline{\Delta} f-\overline{\nabla}^2f(\nu,\nu)-Hg(\overline{\nabla}f,\nu),
\end{equation*}
and $f$-minimal equation $H=g(\overline{\nabla}f,\nu)$. Then we have
\begin{align}\label{q-confor}
  |\tilde{A}|^2_{\tilde{g}}+\widetilde{Ric}(\tilde{\nu},\tilde{\nu}) =&e^{\frac{2f}n}\biggl(|A|_g^2+Ric_f(\nu,\nu)+\frac{1}n\Delta f -\frac{n-1}{n^2}|\nabla f|^2\biggr).
\end{align}

Taking the trace of \eqref{ric-conformal}, we have that the scalar curvature changes after the conformal change of metric by
\begin{equation}\label{scalar-conform}
  \tilde{R}=e^{\frac{2f}n}\left(R+2\overline{\Delta}f-\frac{n-1}n|\overline{\nabla}f|^2\right).
\end{equation}
Recall that the natural scalar curvature in smooth metric measure space is the Perelman's scalar curvature $R_f^{P}$  (see \cite{perel}) defined by
\begin{equation}\label{scalar-perelm}
  R_{f}^{P}=R+2\overline{\Delta}f-|\overline{\nabla}f|^2,
\end{equation}
which is different from $tr(Ric_f)=R+\overline{\Delta}f$. Then combining \eqref{scalar-conform} and \eqref{scalar-perelm}, we have
\begin{equation}\label{scalar-rel}
   \tilde{R}=e^{\frac{2f}n}(R_{f}^{P}+\frac 1n|\overline{\nabla}f|^2).
\end{equation}
Thus if the Perelman's scalar curvature $R_f^{P}$ is nonnegative, the the scalar curvature $\tilde{R}$ of the conformal metric $\tilde{g}$ is also nonnegative.

In particular, for self-shrinkers $\Sigma^n$ in $\R^{n+1}$, (which are $f$-minimal with $f=|x|^2/4$), we have
\begin{align}\label{q-conf-self-s}
  |\tilde{A}|^2_{\tilde{g}}+ \widetilde{Ric}(\tilde{\nu},\tilde{\nu}) =& e^{\frac{|x|^2}{2n}}\biggl( |A|^2_g+1-\frac 1{4n}|x|^2+\frac {1}{4n^2}|x^{\top}|^2\biggr),
\end{align}
where $x^{\top}$ denotes the tangential part of the position vector $x$ of $\Sigma$. From \eqref{scalar-conform} we also have that the scalar curvature of $(\R^{n+1},\tilde{g})$ satisfies
\begin{equation}\label{scalar-self-sr}
  \tilde{R}=e^{\frac{|x|^2}{2n}}\left(n+1-\frac{n-1}{4n}|x|^2\right),
\end{equation}
which appeared in \cite{CM09}. It is easy to see from \eqref{scalar-self-sr} that $\tilde{R}$ tends to negative infinity as $|x|\ra  \infty$.

\section{Proof of main theorem}\label{sec:proof}

We first have the following local singular compactness theorem for the space of embedded minimal hypersurfaces with bounded volume and index.
\begin{prop}\label{prop3-1}
Given a point $p$ in $M^{n+1}$ with $2\leq n\leq 6$. There exists $R>0$ such that the following is true: Suppose $\Sigma_j$ is a sequence of embedded minimal hypersurfaces in $B_{2R}=B_{2R}(p)\subset M$ with $\pt\Sigma_j\subset\pt B_{2R}$. If each $\Sigma_j$ has volume at most $\Lambda$ and index at most $I$ for some fixed $\Lambda>0,I\in \mathbb{N}$, then there exists a smooth embedded minimal hypersurface $\Sigma\subset B_R$ with $\pt\Sigma\subset\pt B_R$, and at most $I$ points $\{x_k\}\subset \Sigma$, and a subsequence of $\Sigma_j$ that converges in $B_R$ with finite multiplicity to $\Sigma$ away from $\{x_k\}$.
\end{prop}
\proof
This proposition is a local version of the second author's \cite{b-sharp} compactness theorem, and the proof is essentially the same. We refer the readers to \cite[\S 4]{b-sharp} for detail of the proof. But for convenience of reader, we would like to recall the key part of the proof to show how the index bound is used.  Basically, the index bound is used to show the limit hypersurface $\Sigma$ is smooth, and the convergence is smooth away from at most $I$ points.

Since each $\Sigma_j$ is minimal and has bounded volume, the Allard's compactness theorem tells us that a subsequence, still denoted by $\Sigma_j$, converges to $\Sigma$ in the varifold sense with $\Sigma$ stationary and integral. Without loss of generality, we can assume that each $\Sigma_j$ has index $I$. Firstly, we can show that $\Sigma\cap B_{2R}$ has at most $I$ singular points. Suppose for contradiction that there has at least $I+1$ singular points $\{x_k\}_{k=1}^{I+1}\subset sing(\Sigma\cap B_{2R})$. Let $\epsilon_0=\frac 12\min \{\min\limits_{k\neq l}d_g(x_k,x_l),inj(M)\}$, where $inj(M)$ denote the injectivity radius of $M$. The convergence is not smooth and graphical over $\Sigma_i\cap B_{\epsilon_0}(x_k)$ for each $k$. Then Schoen-Simon's regularity theorem \cite[Corollary 1]{sch-sim} implies that there is a subsequence $\Sigma_i$ and sufficiently large $i_0$ such that for all $i\geq i_0$, the index of $\Sigma_i\cap B_{\epsilon_0}(x_k)$ is at least one for all $1\leq k\leq I+1$. Since the $I+1$ subsets $\Sigma_i\cap B_{\epsilon_0}(x_k)$ are mutually disjoint. From Lemma 3.1 in \cite{b-sharp}, we have that for each $\Sigma_i$, at least one  $B_{\epsilon_0}(x_k)\cap\Sigma_i$ has index zero, which is a contradiction.

Next, we can use the index bound to show that for each $x_k\in sing(\Sigma\cap B_{2R})$, there exists some $\epsilon_k$ such that $(B_{\epsilon_k}(x_k)\setminus\{x_k\})\cap\Sigma$ is stable (see claim 2 in \cite[\S 4]{b-sharp}). Then
\begin{equation*}
  \mathcal{H}^{n-2}(sing(\Sigma\cap B_{2R}))\left\{\begin{array}{cl}
                                                     =0, &\textrm{ if }n\geq 3, \\
                                                     <\infty, &\textrm{ if }n=2 ,
                                                   \end{array}\right.
\end{equation*}
which by the regularity results of Schoen-Simon \cite[Corollary 1]{sch-sim} (and \cite[Remark 2.2]{b-sharp}) implies that $\Sigma\cap B_{R}$ is smooth.
\endproof

\begin{rem}
In the case of $n=2$, the above proposition recovers Proposition 2.1 in Colding-Minicozzi's paper \cite{CM09}. This follows since the local genus and area bound implies the index bound, thus the conclusion in \cite[Proposition 2.1]{CM09} follows from the above proposition.
\end{rem}

Using Proposition \ref{prop3-1} and a covering argument as in \cite{CM09,CMZ}, we can prove a global singular compactness theorem for complete $f$-minimal hypersurfaces.

\begin{prop}\label{prop3-2}
Let $(M^{n+1},g,e^{-f}d\mu)$ be a complete smooth metric measure space with $2\leq n\leq 6$. Given an integer $I$ and a constant $0<\Lambda<\infty$. Suppose that $\Sigma_j\subset M$ is a sequence of smooth complete embedded $f$-minimal hypersurfaces with $f$-index at most $I$, without boundary and with weighted volume $\int_{\Sigma_j}e^{-f}dv$ at most $\Lambda$. Then there exists a smooth embedded minimal hypersurface $\Sigma\subset M$, a subsequence of $\Sigma_j$ and a locally finite collection of points $\mathcal{S}\subset \Sigma$ such that $\Sigma_j$ converges smoothly (possibly with finite multiplicity) to $\Sigma$ away from $\mathcal{S}$.
\end{prop}
\proof
Recall that an $f$-minimal hypersurface is a minimal hypersurface under the conformal change of the ambient metric, see \S \ref{conf-chang}. Then each $\Sigma_j$ is a minimal hypersurface in $(M^{n+1},\tilde{g})$, where $\tilde{g}=e^{-\frac 2nf}g$ is the conformal metric on $M$.  Let $\tilde{B}_{2R}(p)\subset M$ denote the geodesic ball in $(M,\tilde{g})$ of radius $2R$ centred at some point $p\in M$. Since on each geodesic ball $\tilde{B}_{2R}$, the area of $\Sigma_j\cap\tilde{B}_{2R}$ is at most $\int_{\Sigma_j}d\tilde{v}=\int_{\Sigma_j}e^{-f}dv\leq \Lambda$, and the index of $\Sigma_j\cap\tilde{B}_{2R}$ is at most $I$, then Proposition \ref{prop3-2} follows by covering $M^{n+1}$ with a countable collection of small geodesic balls $\tilde{B}_{2R}$, applying Proposition \ref{prop3-1} and passing to a diagonal subsequence. Moreover, Proposition \ref{prop3-1} tells us the convergence is smooth away from a locally finite set $\mathcal{S}$, in fact on each $\tilde{B}_R$, the number of points $\mathcal{S}\cap \tilde{B}_R$ is at most $I$.
\endproof

Now we can prove our main theorem.
\begin{proof}[Proof of Theorem \ref{main-thm}]
By the assumption $Ric_f\geq \kappa>0$ on $M$, we have that the fundamental group of $M$ is finite, see \cite{morg,Wei-Wylie}. After passing to a covering argument, without loss of generality, we can assume that $M$ is simply connected. Since each $\Sigma_j$ has finite weighted volume, then each $\Sigma_j$ is two-sided, see \cite[Proposition 8]{CMZ}.  From Proposition \ref{prop3-2}; thus to prove our theorem it remains to show the convergence is smooth everywhere. By Allard's regularity theorem, it suffices to check that the convergence has multiplicity one.

If the multiplicity of the convergence is greater than one, by following the same argument as in the proof of Proposition 3.2 in \cite{CM09} (see also \cite{b-sharp}), there exists a smooth positive function $u$ on $\Sigma$ satisfying $L_fu=0$, which by a classical argument (see e.g., \cite[Lemma 1.36]{CM2011}) implies that $\Sigma$ is a two-sided $f$-stable $f$-minimal hypersurface. However, there exists no complete two-sided $f$-stable $f$-minimal hypersurface without boundary and with finite weighted volume in complete smooth metric measure space $(M^{n+1},g,e^{-f}d\mu)$ satisfying $Ric_f\geq \kappa>0$, see \cite[Theorem 3]{CMZ}. This is a contradiction, so the multiplicity of the convergence must be one and the convergence is smooth everywhere.

Finally, we show that the $f$-index of the limit hypersurface $\Sigma$ is at most $I$ by adapting the argument in \cite{b-sharp}. Suppose the $f$-index of $\Sigma$ is at least $I+1$, then there exists a sufficiently large $R_1$ such that the $f$-index of $B_{\rho}(p)$ is at least $I+1$ for all $\rho\geq R_1$, where $B_{\rho}(p)$ is a geodesic ball on $\Sigma$ of radius $\rho$ centred at $p\in\Sigma$. There exist $I+1$ functions $\{u_k\}_{k=1}^{I+1}$ compactly supported in $B_{R_1}$ which are weighted $L^2$ orthogonal, i.e.,
\begin{equation*}
\int_{\Sigma}u_ku_le^{-f}dv=\delta_{kl},\quad k,l=1,\cdots,I+1,
\end{equation*}
such that
\begin{equation*}
  Q_f^{\Sigma}(u_k,u_k)=\lambda_k,\quad \lambda_k<0,\quad k=1,\cdots,I+1.
\end{equation*}
Let $\phi_k$ be an arbitrary $C^1$ extension of $u_k\nu$ to the rest of $M$ and $u_k^i=g(\phi_k,\nu_i)$,  where $\nu_i$ is the unit normal vector of $\Sigma_i$. Then there exists sufficiently large $i_0$ such that for all $i\geq i_0$, we have $Q_f^{\Sigma_i}(u_k^i,u_k^i)\leq \frac 34 \lambda_{I+1}<0$ for each $k$, where we have ordered the eigenvalues such that $\lambda_1<\lambda_2\leq\cdots\leq \lambda_{I+1}<0$. Choose $p_i\in\Sigma_i$ such that $p_i\ra p$ as $i\ra\infty$. Let $\xi(\rho):\R^+\ra\R^+$ be a smooth cut-off function satisfying
\begin{equation*}
  \xi(\rho)=\left\{\begin{array}{cl}
                     1,&\textrm{ for }\rho\in [0,R_2],\\
                     0 &\textrm{ for }\rho\in[2R_2,\infty),
                   \end{array}\right. \quad 0\leq \xi\leq 1,\quad |\xi'|\leq C/{R_2},
\end{equation*}
We can choose $R_2\gg R_1$ sufficiently large such that $Q_f^{\Sigma_i}(\xi u_k^i,\xi u_k^i)\leq \frac 12 \lambda_{I+1}<0$ for each $k$ and all $i\geq i_0$, where we view $\xi(\rho)$ as a composite function of $\xi$ and the distance function $\rho$ on $\Sigma_i$ from a fixed point $p_i$. Note that each $\xi u_k^i$ is compactly supported in $B_{2R_2}(p_i)\subset\Sigma_i$ and the $f$-index of $\Sigma_i$ is at most $I$, $\{\xi u_k^i\}_{k=1}^{I+1}$ must be linearly dependent for all $i\geq i_0$. After passing to a subsequence and re-ordering, without loss of generality, we can assume that
\begin{equation}\label{pf-thm-2}
  \xi u_{I+1}^i=a_1^i\xi u_1^i+\cdots+a_I^i\xi u_I^i,\quad |a_k^i|\leq 1, \quad 1\leq k\leq I,~i\geq i_0.
\end{equation}
Note that
\begin{equation*}
  \int_{\Sigma_i}|\phi_k^{\top_i}|^2e^{-f}dv_{\Sigma_i}\ra 0,\quad\textrm{ as }i\ra \infty,
\end{equation*}
since we know that $\Sigma_i$ converges to $\Sigma$ smoothly everywhere, where $\phi_k^{\top_i}$ means the tangential part of the vector field $\phi_k$ on each $\Sigma_i$. Then
\begin{align}\label{pf-thm-3}
  \lim_{i\ra\infty}\int_{\Sigma_i}\xi^2u_k^iu_l^ie^{-f}dv_{\Sigma_i}=&\lim_{i\ra\infty}\int_{\Sigma_i}\xi^2(g(\phi_k,\phi_l)-g(\phi_k^{\top_i},\phi_l^{\top_i}))e^{-f}dv_{\Sigma_i}\nonumber\\
  =&\int_{\Sigma}u_ku_le^{-f}dv=\delta_{kl},\quad k,l=1,\cdots,I+1,
\end{align}
where in the second equality we used the fact that $u_k$ are compactly supported in $B_{R_1}(p)$ and $\xi\equiv 1$ in $B_{R_2}(p)$.  Using \eqref{pf-thm-2} and \eqref{pf-thm-3}, we have for all $k=1,2,\cdots,I$,
\begin{equation}\label{pf-thm-4}
  0=\lim_{i\ra\infty}\int_{\Sigma_i}\xi^2 u_k^iu_{I+1}^ie^{-f}dv_{\Sigma_i}=\lim_{i\ra\infty}a_k^i.
\end{equation}
Combining \eqref{pf-thm-2}, \eqref{pf-thm-3} and \eqref{pf-thm-4}, we obtain
\begin{align*}
  \lim_{i\ra\infty}\int_{\Sigma_i}(\xi u_{I+1}^i)^2e^{-f}dv_{\Sigma_i}=&\lim_{i\ra\infty}\sum_{k=1}^{I}(a_k^i)^2\int_{\Sigma_i}(\xi u_{k}^i)^2e^{-f}dv_{\Sigma_i}=0,
\end{align*}
which contradicts \eqref{pf-thm-3}. Thus the $f$-index  of $\Sigma$ is at most $I$ and we have completed the proof of Theorem \ref{main-thm}.
\end{proof}

\section{Estimate on the upper bound of the $f$-index}\label{sec:index-bd}

Here we prove some upper bounds on compact $f$-minimal surfaces. We will use results of Ejiri-Micallef (when $n=2$) and Cheng-Tysk (when $n\geq 3$) which together give the following Theorem. 

\begin{thm}[\cite{Cheng-Tysk} \cite{Eiji-M}]\label{thm-min}
Let $\Sigma^n\subset \tilde{M}^m$ be a closed minimally immersed manifold in a closed Riemannian manifold $\tilde{M}$. Then the index of $\Sigma$ satisfies
\begin{equation}\label{min}
ind(\Sigma)\leq \twopartdef{C(n,\tilde{M})(m-n)\int_{\Sigma} \max\{1, |A|^2 + Ric_{\tilde{M}}(\nu,\nu)\}^{\f{n}{2}}dv}{n\geq 3}{C(\tilde{M})m(\int_{\Sigma} dv + \gamma-1)}{n=2,} 
\end{equation}
where $C(\tilde{M})$ is a constant depending on an isometric embedding of $\tilde{M}$ into some Euclidean space. 
\end{thm}
\begin{rem}
We point out here that if $\tilde{M}$ is a compact manifold with boundary, then the above theorem is trivially true. This follows by exactly the same proof as the above authors', except we have to isometrically embed our manifold with boundary into Euclidean space so that the second fundamental form is uniformly bounded in $L^{\infty}$ - but this follows easily since the Nash embedding theorem holds for manifolds with boundary and $\tilde{M}$ is compact. 
\end{rem}

\begin{proof}[Proof of Proposition \ref{1.8}]
Under the assumptions of Proposition \ref{1.8} we can let $(\tilde{M}, \tilde{g}) =(K, e^{-\f{2f}{n}}g)$ and apply Theorem \ref{thm-min} (and the proceeding remark) to give 
\begin{equation}
\widetilde{ind}(\Sigma)\leq \twopartdef{C(n,\tilde{M})\int_{\Sigma} \max\{1, |\tilde{A}|_{\tilde{g}}^2 + \widetilde{Ric}_{\tilde{M}}(\tilde{\nu},\tilde{\nu})\}^{\f{n}{2}}d\tilde{v}}{n\geq 3}{C(\tilde{M})(\int_{\Sigma} d\tilde{v} + \gamma-1)}{n=2.} 
\end{equation}
It is an easy exercise using \eqref{q-confor} to finish the proof. 
\end{proof}

We will now prove an estimate on the $f$-index of an $f$-minimal hypersurface which is not necessarily compact with $n\geq 3$. We do this by following the methods of Cheng-Tysk \cite{Cheng-Tysk}. Firstly, we recall a Sobolev inequality in the weighted setting proved by Impera-Rimoldi \cite{Impera-Ri}.
\begin{thm}[\cite{Impera-Ri}]\label{sob-thm}
Let $\Sigma^n$ be an immersed hypersurface in $(M^{n+1},g,e^{-f}d\mu)$. Suppose that the sectional curvature $sect(M)\leq 0$ and there exists a constant $c_n>0$ such that
\begin{equation}\label{sob-condi}
  \limsup_{\rho\ra 0^+}\frac{Vol_f(S_{\rho}(\xi))}{\rho^n}\geq c_n,
\end{equation}
for all $\xi\in \Sigma$, where $S_{\rho}(\xi)$ denotes
\begin{equation*}
 S_{\rho}(\xi)=\{x\in\Sigma|\textrm{dist}_M(x,\xi)\leq \rho\}.
\end{equation*}
Then there exists a constant $C_n$ such that for any nonnegative compactly supported $C^1$ function $\varphi$ on $\Sigma$, we have
\begin{equation}\label{sobolev-weighted}
  \left(\int_{\Sigma}\varphi^{\frac n{n-1}}e^{-f}dv\right)^{\frac {n-1}n}\leq C_n\int_{\Sigma}\left(|\nabla \varphi|+\varphi(|H_f|+|\overline{\nabla}f|)\right)e^{-f}dv.
\end{equation}
\end{thm}

The condition \ref{sob-condi} is satisfied if we assume that $\sup_{\Sigma}f<\infty$, since
\begin{equation*}
  \limsup_{\rho\ra 0^+}\frac{Vol(S_{\rho}(\xi))}{\rho^n}\geq \omega_n,
\end{equation*}
for all $\xi\in \Sigma$, where $\omega_n$ denotes the volume of $n$-dimensional unit ball in $\R^n$. We remark that Batista-Mirandola \cite{Bat-Miran} also obtained a similar Sobolev inequality in the weighted setting independently.

The proof of Theorem \ref{thm-min} in \cite{Cheng-Tysk} used the heat kernel and Michael-Simon Sobolev inequality, inspired by a similar argument in Li-Yau's paper \cite{Li-yau}. In the following, we will adapt the idea in \cite{Cheng-Tysk} to obtain an upper bound on the $f$-index of $f$-minimal hypersurface.

\begin{thm}\label{thm-index-up}
Suppose $n\geq 3$, and let $\Sigma^n $ be a smooth complete $f$-minimal hypersurface in $(M^{n+1},g,e^{-f}d\mu)$, where $M$ satisfies $sect(M)\leq 0$ and \eqref{sob-condi} as in Theorem \ref{sob-thm},  and $|\overline{\nabla}f|\leq c_1$ on $\Sigma$. Then the $f$-index of $\Sigma$ satisfies
\begin{equation}\label{indx-upperbd}
  ind_f(\Sigma)\leq C(n,M,c_1)\int_{\Sigma}(\max\{\frac 12,|A|^2+Ric_f(\nu,\nu)\})^{\frac n2}e^{-f}dv,
\end{equation}
where $C(n,M,c_1)$ is a uniform constant depending on $n,M,c_1$.
\end{thm}
\proof
The proof is similar as that in \cite[Theorem 1]{Cheng-Tysk} (see also \cite{Li-yau}), so we just sketch it. Assume that $\Sigma$ is orientable, otherwise we pass to the orientable double cover. Recall that the $L_f$ operator is
\begin{equation*}
  L_f=\Delta_f+|A|^2+Ric_f(\nu,\nu).
\end{equation*}
Consider the weighted manifold $(\Sigma^n,g,e^{-f}dv)$. Denote $p(x)=\max\{\frac 12,|A|^2+Ric_f(\nu,\nu)\}$ and $D\subset\Sigma$ any compact domain. We know that (see \cite{Grig06}) there exists a unique heat kernel $H(x,y,t)$, a $C^{\infty}$ function on $\Sigma\times\Sigma\times\R^+$,  of the operator
\begin{equation*}
  \frac{\pt }{\pt t}-\frac 1{p(x)}\Delta_f
\end{equation*}
on $D$, which satisfies Dirichlet boundary condition if $\pt D\neq {\O}$. Let $\{\lambda_i\}_{i=0}^{\infty}$ be the Dirichlet eigenvalues of $\frac 1{p(x)}\Delta_f$ on $D$. Define
\begin{equation*}
  h(t)=\sum_{i=0}^{\infty}e^{-2\lambda_it}.
\end{equation*}
Then
\begin{equation*}
  h(t)=\int_D\int_D H^2(x,y,t)p(x)p(y)e^{-f(x)}dv(x)e^{-f(y)}dv(y).
\end{equation*}
Differentiation in $t$ and integration by parts yield
\begin{align*}
  \frac{dh}{dt}= & -2\int_D\int_D|\nabla_yH(x,y,t)|^2p(x)e^{-f(x)}dv(x)e^{-f(y)}dv(y).
\end{align*}
Using the H\"{o}lder inequality and the property of heat kernel, by arguing as the proof of Theorem 1 in \cite{Cheng-Tysk}, we have
\begin{align}\label{thm4-2-pf1}
  &h^{\frac{n+2}n}(t) (\int_Dp(x)^{\frac n2}e^{-f(x)}dv(x))^{-\frac 2n} \nonumber\\
   & \qquad \leq \int_Dp(x)(\int_DH^{\frac{2n}{n-2}}(x,y,t)e^{-f(y)}dv(y))^{\frac {n-2}n}e^{-f(x)}dv(x).
\end{align}
Replacing $\varphi=H^{\frac{2(n-1)}{n-2}}$ in the weighted Sobolev inequality \eqref{sobolev-weighted}, and applying the H\"{o}lder-inequality, we obtain
\begin{equation}\label{sobolev-2}
  \left(\int_{\Sigma}H^{\frac {2n}{n-2}}e^{-f}dv\right)^{\frac {n-2}n}\leq C\int_{\Sigma}\left(|\nabla H|^2+c_1^2H^2\right)e^{-f}dv,
\end{equation}
where we used the fact $H_f=0$ since $\Sigma$ is $f$-minimal; and $|\overline{\nabla}f|\leq c_1$ on $M$. Applying \eqref{sobolev-2} to \eqref{thm4-2-pf1}, we obtain
\begin{align}\label{thm4-2-pf2}
  h^{\frac{n+2}n}(t) (\int_Dp(x)^{\frac n2}e^{-f(x)}dv(x))^{-\frac 2n}\leq -\frac 12C\frac{dh}{dt}+2Cc_1^2h(t),
\end{align}
where we used the fact $p(y)\geq 1/2$ for all $y\in\Sigma$. From the differential inequality \eqref{thm4-2-pf2}, using the fact that $h(t)\ra\infty$ as $t\ra 0^+$, we have
\begin{align*}
h(t)\leq &\frac{(2C)^{\frac n2}c_1^n}{(1-e^{-\frac{8t}nc_1^2})^{\frac n2}}\int_{\Sigma}p(x)^{\frac n2}e^{-f(x)}dv(x).
\end{align*}
As in \cite{Cheng-Tysk,Li-yau}, we have
\begin{align*}
ind_f(D)\leq &\inf_{t>0}e^{2t}h(t)\leq  C\int_{\Sigma}p(x)^{\frac n2}e^{-f(x)}dv(x),
\end{align*}
where $C$ is a constant independent of $D$. Since $ind_f(\Sigma)= \sup_{D\subset\Sigma}ind_f(D)$, this finishes the proof of Theorem \ref{thm-index-up}.
\endproof
\begin{rem}
From the formula \eqref{q-confor}, we can see that although $f$-minimal hypersurface can be viewed as a minimal hypersurface after the conformal change of the ambient metric, the conclusion in Theorem \ref{thm-index-up} can not follow directly from Cheng-Tsyk's result \cite{Cheng-Tysk} after the conformal transformation. The conditions on $M$ and $f$ assumed in Theorem \ref{thm-index-up} are only used to guarantee that the Sobolev inequality \eqref{sobolev-2} holds.  Thus the index upper bound \eqref{indx-upperbd} also holds for $f$-minimal hypersurfaces $\Sigma$ in $(M^{n+1},g,e^{-f}d\mu)$ if \eqref{sobolev-2} holds on $\Sigma$.
\end{rem}

Theorem \ref{thm-index-up} and Proposition \ref{1.8} do not cover the general complete non-compact $f$-minimal hypersurfaces and self-shrinkers. We know that if a self-shrinker $\Sigma^n\subset\R^{n+1}$ arises from the Type-I singularity of mean curvature flow, then the second fundamental form of $\Sigma$ must be bounded and  $\Sigma$ has at most Euclidean volume growth. Then $\Sigma$ has finite weighted volume, and the right-hand side of \eqref{indx-upperbd-selfs1} is bounded above. Thus, if the conclusion in Proposition \ref{mainprop-indexup-ss1} also holds for general complete non-compact self-shrinker,  then we can obtain an upper bound on the $f$-index of the self-shrinker arising from Type-I singularity, and Corollary \ref{main-corr} would imply that the space of Type-I singularities of the mean curvature flow is compact in smooth topology with multiplicity one. About the lower bound of the $f$-index, there is one result by Hussey \cite{hussey} for self-shrinkers. In his thesis, Hussey proved that a flat hyperplane $\R^{n}$ through the origin, as a self-shrinker, has index 1; the cylinder $\mathbb{S}^k(\sqrt{2k})\times\R^{n-k}, 1\leq k\leq n$ has index $n+2$; any other smooth complete embedded self-shrinker without boundary and with polynomial volume growth has index at least $n+3$. See \cite{hussey} for details. We will investigate more estimates (upper bound and lower bound) on the $f$-index for $f$-minimal hypersurfaces, in particular for self-shrinkers in the future.

\section{$f$-minimal surfaces of finite $f$-index  in $(M^{3},g,e^{-f}d\mu)$}

In \cite{fisch-scho}, Fischer-Colbrie and Schoen showed that a complete stable minimal
surface $\Sigma$ in a complete manifold $(M^3,g)$ with nonnegative scalar curvature must
be either conformally the complex plane $\mathbb{C}$ or conformally a cylinder $\R\times \SS^1$. This was generalised by Espinar \cite{Espinar} to $f$-stable $f$-minimal surface $\Sigma$ in a smooth metric measure space $(M^{3},g,e^{-f}d\mu)$ with nonnegative Perelman's scalar curvature $R_f^{P}$, i.e. such $\Sigma$ is either conformally the complex plane $\mathbb{C}$ or conformally a cylinder $\R\times \SS^1$. To show this, he introduced the following key lemma:
\begin{lem}[\cite{Espinar}]\label{lem-espinar}
Let $\Sigma$ be a complete orientable $f$-stable $f$-minimal surface in $(M^{3},g,e^{-f}d\mu)$. Then
\begin{equation}
  \int_{\Sigma}(V-\frac 13K)\varphi^2dv\leq \int_{\Sigma}|\nabla\varphi|^2dv,
\end{equation}
for any $\varphi\in C_c^{\infty}(\Sigma)$, where $K$ is the Gaussian curvature of $\Sigma$ with respect to the induced metric from $(M^3,g)$ which we still denote by $g$, and
\begin{equation}\label{V-def}
  V=\frac 13\left(\frac 12R_f^P+\frac 12|A|^2+\frac 18|\nabla f|^2\right).
\end{equation}
In other words, the Schr\"{o}dinger operator $L=\Delta-\frac 13K+V$ is stable on $\Sigma$.
\end{lem}

We remark that in \cite{Liu}, Liu constructed an example to show that Fischer-Colbrie-Schoen's \cite{fisch-scho} result can not be generalised to the $f$-minimal surface in $(M^{3},g,e^{-f}d\mu)$ with $tr(Ric_f)\geq 0$.

In this section, we will consider the conformal property of the complete orientable $f$-minimal surfaces of finite $f$-index  in $(M^{3},g,e^{-f}d\mu)$. Recall that in \cite[Theorem 1]{fisch}, Fischer-Colbrie proved that a complete orientable minimal surface $\Sigma$ of finite index in a three manifold $(M^3,g)$ with nonnegative scalar curvature is conformally diffeomorphic to a Riemannian surface with a finite number of points removed. We will prove

\begin{thm}\label{thm-fisch}
Let $\Sigma$ be a complete orientable $f$-minimal surface of finite $f$-index in $(M^{3},g,e^{-f}d\mu)$ with $V\geq 0$, then there exists a complete conformal metric on $\Sigma$ with nonnegative Gaussian curvature outside a compact set. In particular, $\Sigma$ is conformally diffeomorphic to a Riemannian surface with a finite number of points removed.
\end{thm}
\proof
Since $\Sigma$ has finite $f$-index, then there exists a compact set $C$ in $\Sigma$ such that $\Sigma\setminus C$ is $f$-stable, see \cite[Proposition 1]{fisch}. From Lemma \ref{lem-espinar}, we know that the Schr\"{o}dinger operator $L=\Delta-\frac 13K+V$ is stable on $\Sigma\setminus C$. Then again by Fischer-Colbrie \cite[Proposition 1]{fisch}, there exists a positive function $u$ on $\Sigma$ such that $Lu=0$ on $\Sigma\setminus C$. As in \cite{Espinar}, we consider the conformal metric $\hat{g}=u^6g$ on $\Sigma$. The Gaussian curvature $\widehat{K}$ of $\hat{g}$ satisfies
\begin{align}\label{K-hat}
  \widehat{K} =& u^{-6}(K-3\Delta \log u)\nonumber \\
  = &u^{-6}(K-3\frac{\Delta u}u+3|\nabla \log u|^2) \nonumber\\
  =&3u^{-6}(V+|\nabla\log u|^2),
\end{align}
where in the last equality we used $0=Lu=\Delta u-\frac 13Ku+Vu$. By the assumption $V\geq 0$, we have $\widehat{K}\geq 0$ on $\Sigma\setminus C$.

We next show that the metric $\hat{g}=u^6g$ is a complete metric on $\Sigma$. Let $\gamma(t): [0,\infty)\ra \Sigma\setminus C$ be a minimising geodesic with respect to the metric $\hat{g}$, where $t$ is the arc-length w.r.t $g$ ($\gamma(t)$ can be obtained as \cite{fisch}). To show that $\hat{g}$ is complete, it suffices to show $\gamma$ has infinite length w.r.t $\hat{g}$, i.e., we need to show $\int_0^{\infty}u^3(\gamma(s))ds=\infty$. Since $\gamma$ is a minimising geodesic, the second variation formula for arc-length gives us
\begin{equation}\label{2nd-var-length}
  \int_0^{\infty}\left((\frac {d\phi}{d\hat{s}})^2-\widehat{K}\phi^2\right)d\hat{s}\geq 0
\end{equation}
for any smooth compactly supported function on $(0,\infty)$. Since $d\hat{s}=u^3ds$, from \eqref{K-hat} and \eqref{2nd-var-length}, we obtain
\begin{equation}\label{thm-fis-pf1}
  3\int_0^{\infty}u^{-5}(s)u'(s)^2\phi(s)^2ds\leq \int_0^{\infty}u^{-3}(s)\phi'(s)^2ds.
\end{equation}
Let $\phi=u^3\psi$ where $\psi$ has compact support on $(0,\infty)$. Then \eqref{thm-fis-pf1} yields
\begin{equation}\label{thm-fis-pf2}
  \int_0^{\infty}u(s)^3\left(\psi'(s)^2+2\psi(s)\psi''(s)\right)ds\leq 0.
\end{equation}
As in \cite{fisch}, we choose the test function $\psi(s)=s\xi(s)$, where $\xi(s)$ satisfies
\begin{equation*}
  \xi(s)=1,\quad s\in [0,R];\qquad \xi(s)=0, \quad s\in(2R,\infty)
\end{equation*}
and $|\xi'|\leq c/R, |\xi''|\leq c/{R^2}$ for $s\in [R,2R]$. Then for this $\psi(s)$, \eqref{thm-fis-pf2} yields
\begin{equation}\label{thm-fis-pf3}
  \int_0^{R}u(s)^3ds\leq c\int_R^{\infty}u(s)^3ds,
\end{equation}
where $c$ is a constant independent of $R$. \eqref{thm-fis-pf3} then implies that $\int_0^{\infty}u(s)^3ds=\infty$, and the metric $\hat{g}=u^3g$ is a complete metric on $\Sigma$ with nonnegative Gaussian curvature $\widehat{K}$ on $\Sigma\setminus C$.

Finally, from the last part in the proof of Theorem 1 in \cite{fisch},  we conclude that $\Sigma$ is conformally diffeomorphic to a Riemannian surface with a finite number of points removed.
\endproof
\begin{rem}
From the definition of $V$ in \eqref{V-def}, if the Perelman's scalar curvature $R_f^P\geq 0$, then $V\geq 0$. Thus Theorem \ref{thm-fisch} is clearly true if we replace the condition $V\geq 0$ by a stronger condition $R_f^P\geq 0$. However from the equation \eqref{scalar-rel}, we know that if $R_f^P\geq 0$, then the scalar curvature of the conformal metric satisfies $\tilde{R}\geq 0$. Therefore for $R_f^P\geq 0$ case, the conclusion of Theorem \ref{thm-fisch} is a direct consequence from Fischer-Colbrie's \cite{fisch} result by a conformal transformation.
\end{rem}
\begin{rem}
Our Theorem \ref{thm-fisch} is nontrivial, since from $V\geq 0$ we cannot say $\tilde{R}\geq 0$.
\end{rem}

\bibliographystyle{Plain}

\end{document}